\documentclass[10pt]{amsart}

\newtheorem{thm}{Theorem}[section]

\newtheorem{obs}{Remark}[section]

\usepackage{amssymb}
\usepackage{amsmath}
\usepackage{amsfonts}
\usepackage{graphicx}
\usepackage{amsthm}

\numberwithin{equation}{section}

\usepackage[T1]{fontenc}
\usepackage{palatino}

\begin{document}
\title[3D enstrophy cascade]
{Coherent vortex structures and 3D enstrophy cascade}
\author{R. Dascaliuc}
\address{Department of Mathematics\\
Oregon State University\\ Corvallis, OR 97332}
\author{Z. Gruji\'c}
\address{Department of Mathematics\\
University of Virginia\\ Charlottesville, VA 22904}
\date{\today}
\begin{abstract}
Existence of 2D enstrophy cascade in a suitable mathematical
setting, and under suitable conditions compatible with 2D turbulence
phenomenology, is known both in the Fourier and in the physical
scales. The goal of this paper is to show that the same
\emph{geometric} condition preventing the formation of singularities
-- $\frac{1}{2}$-H\"older coherence of the vorticity direction --
coupled with a suitable condition on a modified Kraichnan scale, and
under a certain modulation assumption on evolution of the vorticity,
leads to existence of 3D enstrophy cascade in physical scales of the
flow.
\end{abstract}

\maketitle

\section{Introduction}

The vorticity-velocity formulation of the 3D Navier-Stokes equations
(NSE) reads
\begin{equation}\label{vorticity}
\partial_t \omega + (u \cdot \nabla)\omega = (\omega \cdot \nabla)u
+
\triangle \omega
\end{equation}
where $u$ is the velocity and $\omega = \, \mbox{curl} \, u$ is the
vorticity of the fluid (the viscosity is set to 1). A comprehensive
introduction to the mathematical study of the vorticity can be found
in \cite{MB02}.

\medskip

In 2D, the vortex-stretching term, $(\omega \cdot \nabla)u$, is
identically zero, and the nonlinearity is simply a part of the
transport of the vorticity by the flow.

\medskip

This allowed the authors to adopt to 2D a general mathematical
setting for the study of turbulent cascades and locality in
\emph{physical scales} of 3D incompressible viscous and inviscid
flows introduced in \cite{DaGr11-1} and \cite{DaGr11-2},
respectively, to establish existence of the enstrophy cascade and
locality in 2D \cite{DaGr11-3}.

\medskip

Since the vortex-stretching term is not a flux-type term, the only
way to establish existence of the enstrophy cascade in 3D in this
setting is to show that its contribution to the ensemble averaging
process can be suitably interpolated between integral-scale averages
of the enstrophy and the enstrophy dissipation rate.

\medskip

This is where \emph{coherent vortex structures} come into play. A
role of the coherent vortex structures in turbulent flows was
recognized as early as the 1500's in Leonardo da Vinci's ``deluge''
drawings. On the other hand, Kolmogorov's K41 phenomenology
\cite{Ko41-1, Ko41-2} does not discern geometric structures; the K41
eddies are essentially amorphous. As stated by Frisch in his book
\emph{Turbulence}, The Legacy of A.N. Kolmogorov \cite{Fr95}, ``Half
a century after Kolmogorov's work on the statistical theory of fully
developed turbulence, we still wonder how his work can be reconciled
with Leonardo's half a millennium old drawings of eddy motion in the
study for the elimination of rapids in the river Arno.'' This remark
was followed by a discussion on dynamical role, as well as
statistical signature of vortex filaments in turbulent flows.
Several (by now classical) directions in the study of the vortex
dynamics of turbulence -- both in 2D and 3D -- are presented in
Chorin's book \emph{Vorticity and Turbulence} (cf. \cite{Ch94} and
the references therein). The approach exposed in \cite{Ch94} is
essentially discrete (probabilistic lattice models); on the other
hand, the first rigorous continuous statistical theory of vortex
filaments was given by P.-L. Lions and Majda in \cite{LM00}.

\medskip

Local anisotropic behavior of the enstrophy, i.e., self-organization
of the regions of high vorticity in coherent vortex structures --
most notably vortex filaments/tubes -- is ubiquitous. In particular,
local alignment or anti-alignment of the vorticity direction, i.e.
\emph{local coherence}, is prominently featured in turbulent flows.
A strong numerical evidence, as well as several theoretical
arguments explaining the physical mechanism behind the formation of
coherent structures -- including rigorous estimates on the flow
directly from the 3D NSE -- can be found, e.g., in \cite{Co90,
SJO91, CPS95, GGH97, GFD99, Oh09, GM11}. One way to look at the
phenomenon of local coherence of the vorticity direction is to
interpret it as a manifestation of the general observation that the
regions of high fluid intensity are -- in the vorticity formulation
-- \emph{locally} `quasi 2D-like' (in 2D, the vorticity direction is
\emph{globally} parallel or antiparallel).

\medskip

The rigorous study of \emph{geometric depletion of the nonlinearity}
in the 3D NSE was pioneered by Constantin when he derived a singular
integral representation for the stretching factor in the evolution
of the vorticity magnitude featuring a geometric kernel depleted by
coherence of the vorticity direction (cf. \cite{Co94}). This was
followed by the paper \cite{CoFe93} where Constantin and Fefferman
showed that as long as the vorticity direction in the regions of
intense vorticity is Lipschitz-coherent, no finite-time blow up can
occur, and later by the paper \cite{daVeigaBe02} where Beirao da
Veiga and Berselli scaled the coherence strength needed to deplete
the nonlinearity down to $\frac{1}{2}$-H\"older.

\medskip

A full \emph{spatiotemporal localization} of the vorticity
formulation of the 3D NSE was developed by one of the authors and
the collaborators in \cite{GrZh06, Gr09, GrGu10-1, GrGu10-2}. The
main obstacle to the full localization of the evolution of the
enstrophy was the spatial localization of the vortex-stretching
term, $(\omega \cdot \nabla) u$. An explicit local representation
for the vortex-stretching term was given in \cite{Gr09}, and the
leading order term reads
\begin{equation}\label{vstloc}
P.V. \int_{B(x_0,2R)} \epsilon_{jkl} \frac{\partial^2}{\partial x_i
\partial y_k} \frac{1}{|x-y|} \phi(y,t) \omega_l(y,t) \, dy \
\phi(x,t) \omega_i(x,t) \omega_j(x,t)
\end{equation}
where $\epsilon_{jkl}$ is the Levi-Civita symbol and $\phi$ is a
spatiotemporal cut-off associated with the ball $B(x_0,R)$. The key
feature of (\ref{vstloc}) is that it displays both \emph{analytic}
(via a local non-homogeneous Div-Curl Lemma \cite{GrGu10-2}) and
\emph{geometric} (via coherence of the vorticity direction
\cite{Gr09, GrGu10-1}) \emph{cancelations}, inducing analytic and
geometric \emph{local depletion of the nonlinearity} in the 3D
vorticity model. (For a different approach to localization of the
vorticity-velocity formulation see \cite{ChKaLe07}.)

\medskip

The present work is envisioned as a contribution to the effort of
understanding the role that the geometry of the flow and in
particular, coherent vortex structures, plays in the theory of
turbulent cascades. More precisely, we show -- utilizing the
aforementioned localization within the general mathematical
framework for the study of turbulent cascades in physical scales of
incompressible flows introduced in \cite{DaGr11-1} (in this case,
via suitable ensemble-averaging of the local enstrophy equality) --
that \emph{coherence of the vorticity direction}, coupled with a
suitable condition on a modified Kraichnan scale, and under a
certain modulation assumption on evolution of the vorticity, leads
to existence of 3D enstrophy cascade in physical scales of the flow.
This furnishes a mathematical evidence that, in contrast to 3D
energy cascade, 3D enstrophy cascade is locally \emph{anisotropic},
providing a form of a reconciliation between Leonardo's and
Kolmogorov's views on turbulence on the \emph{enstrophy level}.

\medskip

It is worth pointing out that our theory of turbulent cascades in
physical scales of 3D incompressible flows is on the \emph{energy
level} \cite{DaGr11-1, DaGr11-2, DaGr12-1} consistent both with the
K41 theory of turbulence and the Onsager's predictions on existence
of the \emph{inviscid} energy cascade (and consequently, with the
phenomena of \emph{anomalous dissipation} and \emph{dissipation
anomaly}), as well as with the previous rigorous mathematical work
on existence of the energy cascade in the wavenumbers \cite{FMRT01}.
In particular, on the energy level, it does not `see' geometric
structures. It is only here, i.e., on the enstrophy level, that the
role of coherent vortex structures is revealed. A distinctive
feature of our theory that makes incorporating the geometry of the
flow and in particular, local coherence possible is the fact that
the cascade takes place in the actual physical scales of the flow.
To the best of our knowledge, there is nothing in the K41 theory
that would contradict existence of the 3D enstrophy cascade; on the
other hand, given that K41 takes place primarily in the Fourier
space, i.e., in the wavenumbers, formulating conditions (within K41)
faithfully reflecting various geometric properties of the flow is a
much more challenging enterprise.

\medskip

The paper is organized as follows. Section 2 recalls the
ensemble-averaging process introduced in \cite{DaGr11-1}, and
Section 3 the spatiotemporal localization of the evolution of the
enstrophy developed in \cite{GrZh06, Gr09, GrGu10-1, GrGu10-2}.
Existence and locality of anisotropic 3D enstrophy cascade is
presented in Section 4.

\section{Ensemble averages}

In studying a PDE model, a natural way of actualizing a concept of
scale is to measure distributional derivatives of a quantity with
respect to the scale. Let $x_0$ be in $B(0,R_0)$ ($R_0$ being the
\emph{integral scale}, $B(0,2R_0)$ contained in $\Omega$ where
$\Omega$ is the global spatial domain) and $0< R \le R_0$.
Considering a locally integrable physical density of interest $f$ on
a ball of radius $2R$, $B(x_0, 2R)$, a \emph{local physical scale
$R$} -- associated to the point $x_0$ -- is realized via bounds on
distributional derivatives of $f$ where a test function $\psi$ is a
refined -- smooth, non-negative, equal to 1 on $B(x_0, R)$ and
featuring optimal bounds on the derivatives over the outer $R$-layer
-- cut-off function on $B(x_0, 2R)$. More explicitly,

\begin{equation}\label{ps}
|(D^\alpha f, \psi)| \le \int_{B(x_0, 2R)} |f| |D^\alpha \psi| \le
\Bigl(c(\alpha) \frac{1}{R^{|\alpha|}} |f| ,
\psi^{\delta(\alpha)}\Bigr)
\end{equation}
for some $c(\alpha)>0$ and $\delta(\alpha)$ in $(0,1)$. (This is
reminiscent of Bernstein inequalities in the Littlewood-Paley
decomposition of a tempered distribution.)

\medskip

Henceforth, we utilize {\em refined} spatiotemporal cut-off
functions $\phi=\phi_{x_0,R,T}=\psi\,\eta$, where $\eta=\eta_T(t)\in
C^\infty (0,T)$ and $\psi=\psi_{x_0,R}(x)\in\mathcal{D}(B(x_0,2R))$
satisfying
\begin{equation}\label{eta_def}
0\le\eta\le1,\quad\eta=0\ \mbox{on}\ (0,T/3),\quad\eta=1\ \mbox{on}\
(2T/3,T),\quad\frac{|\eta'|}{\eta^{\rho_1}}\le\frac{C_0}{T}\;
\end{equation}
and
\begin{equation}\label{psi_def}
0\le\psi\le 1,\quad\psi=1\ \mbox{on}\ B(x_0,R),
\quad\frac{|\nabla\psi|}{\psi^{\rho_2}}\le\frac{C_0}{R}, \quad
\frac{|\triangle\psi|}{\psi^{2\rho_2-1}}\le\frac{C_0}{R^2}\;,
\end{equation}
for some $\frac{1}{2} <\rho_1,\rho_2 < 1$.

In particular, $\phi_0=\psi_0 \eta$ where $\psi_0$ is the spatial
cut-off (as above) corresponding to $x_0=0$ and $R=R_0$.

For $x_0$ near the boundary of the integral domain, $S(0,R_0)$, we
assume additional conditions,
\begin{equation}\label{psi_bd}
0\le\psi\le\psi_0
\end{equation}
and, if  $B(x_0,R)\not\subset B(0,R_0)$, then
$\psi\in\mathcal{D}(B(0,2R_0))$ with $\psi=1\ \mbox{on}\ B(x_0,R)
\cap B(0,R_0)$ satisfying, in addition to (\ref{psi_def}), the
following:
\begin{equation}\label{psi_def_add1}
\begin{aligned}
&
\psi=\psi_0\ \mbox{on the part of the cone centered at zero and passing through}\\
& S(0,R_0)\cap B(x_0,R)\ \mbox{between}\  S(0,R_0)\
\mbox{and}\ S(0,2R_0)
\end{aligned}
\end{equation}
and
\begin{equation}\label{psi_def_add2}
\begin{aligned}
&
\psi=0\ \mbox{on}\ B(0,R_0)\setminus B(x_0,2R)\ \mbox{and outside the part of the cone}\\
 &
 \mbox{centered at zero and passing through}\ S(0,R_0)\cap B(x_0,2R)\\
 &
 \mbox{between}\  S(0,R_0)\ \mbox{and}\ S(0,2R_0).
\end{aligned}
\end{equation}

\medskip

A \emph{physical scale $R$} -- associated to the integral domain
$B(0,R_0)$ -- is realized via suitable ensemble-averaging of the
localized quantities with respect to `$(K_1,K_2)$-covers' at scale
$R$.

\medskip

Let $K_1$ and $K_2$ be two positive integers, and $0 < R \le R_0$. A
cover $\{B(x_i,R)\}_{i=1}^n$ of the integral domain $B(0,R_0)$ is a
\emph{$(K_1,K_2)$-cover at scale $R$} if
\[
 \biggl(\frac{R_0}{R}\biggr)^3 \le n \le K_1
 \biggr(\frac{R_0}{R}\biggr)^3,
\]
and any point $x$ in $B(0,R_0)$ is covered by at most $K_2$ balls
$B(x_i,2R)$. The parameters $K_1$ and $K_2$ represent the maximal
\emph{global} and \emph{local multiplicities}, respectively.

\medskip

For a physical density of interest $f$, consider time-averaged, per
unit mass -- spatially localized to the cover elements $B(x_i, R)$
-- local quantities $\hat{f}_{x_i,R}$,
\[
\hat{f}_{x_i,R} = \frac{1}{T} \int_0^T \frac{1}{R^3}
\int_{B(x_i,2R)} f(x,t) \phi^\delta_{x_i,R,T} (x,t) \, dx \, dt
\]
for some $0 < \delta \le 1$, and denote by $\langle F\rangle_R$ the
\emph{ensemble average} given by
\[
 \langle F\rangle_R = \frac{1}{n} \sum_{i=1}^n
 \hat{f}_{x_i,R}\,.
\]

\medskip

The key feature of the ensemble averages $\{\langle
F\rangle_R\}_{0<R\le R_0}$ is that $\langle F\rangle_R$ being
\emph{stable}, i.e., nearly-independent on a particular choice of
the cover (with the fixed parameters $K_1$ and $K_2$), indicates
there are no significant fluctuations of the sign of the density $f$
at scales comparable or greater than $R$. On the other hand, if $f$
does exhibit significant sign-fluctuations at scales comparable or
greater than $R$, suitable rearrangements of the cover elements up
to the maximal multiplicities -- emphasizing first the positive and
then the negative parts of the function -- will result in $\langle
F\rangle_R$ experiencing a wide range of values, from positive
through zero to negative, respectively.

\medskip

Consequently, for an \emph{a priori} sign-varying density, the
ensemble averaging process acts as a \emph{coarse detector of the
sign-fluctuations at scale $R$}. (The larger the maximal
multiplicities $K_1$ and $K_2$, the finer detection.)

\medskip

As expected, for a non-negative density $f$, all the averages are
comparable to each other throughout the full range of scales $R$, $0
< R \le R_0$; in particular, they are all comparable to the simple
average over the integral domain. More precisely,
\begin{equation}\label{k*}
 \frac{1}{K_*} F_0 \le \langle F \rangle_R \le K_* F_0
\end{equation}
for all $0 < R \le R_0$, where
\[
 F_0=\frac{1}{T}\int \frac{1}{R_0^3} \int  f(x,t)
 \phi_0^\delta (x,t) \, dx \, dt,
\]
and $K_* = K_*(K_1,K_2) > 1$.

\medskip

There are several properties of the ensemble averaging process that
although being plausible, deserve a precise analytic
description/quantification. Besides the above statement on
non-negative densities, perhaps the most elemental one is that if
the averages at a certain scale are nearly independent of a
particular choice of a $(K_1,K_2)$-cover ($K_1$ and $K_2$ fixed),
then essentially the same \emph{universality property} should
propagate to larger scales. In order to obtain a precise
quantitative propagation result, the universality is assumed on an
initial interval, rather than at a single scale (this is due to the
presence of smooth cut-offs with prescribed rates of change). Two
types of results are presently available.

\medskip

Let $f$ be a locally integrable function (a density), and $K_1$ and
$K_2$ two positive integers.

\medskip

\noindent TYPE I. \ Assume that there exists $R_* > 0$ such that for
any $R$ in $[R^*, 2R_*)$ and any $(27K_1,16K_2)$-cover at scale $R$,
the averages $\langle F\rangle_R$ are all comparable to some value
$F_*$; more precisely, $\displaystyle{\frac{1}{C_1} F_* \le \langle
F\rangle_R \le C_1 F_*}$. Then, for all $R \ge 2R_*$ and all
$(K_1,K_2)$-covers at scale $R$, the averages $\langle F\rangle_R$
satisfy $\displaystyle{\frac{1}{C_2} F_* \le \langle F\rangle_R \le
C_2 F_*}$.

\medskip

\noindent TYPE II. \ Assume that there exists $R_* > 0$ such that
for any $R$ in $(\frac{1}{2}R^*, \frac{5}{2}R_*)$ and any
$(K_1,K_2)$-cover at scale $R$, $\displaystyle{\frac{1}{C_1} F_* \le
\langle F\rangle_R \le C_1 F_*}$. Then, for all $R \ge \frac{5}{2}
R_*$ and all $(K_1,K_2)$-covers at scale $R$, the averages $\langle
F\rangle_R$ satisfy

\[
\frac{1}{C_3} \Bigl(\frac{R_*}{R}\Bigr)^{C_4} F_* \le \langle
F\rangle_R \le C_3 \Bigl(\frac{R}{R^*}\Bigr)^{C_4} F_*.
\]

\medskip

Shortly, in a Type I result, the universality is assumed with
respect to more refined covers, while in a Type II result, the
non-exactness of the propagation caused by the smooth cut-offs is
reflected in a correction to the universal value $F_*$ by the ratio
of the scales $R$ and $R^*$.

\medskip

The proofs are quite long and technical, and will be provided in a
separate publication together with computational results describing
the general statistics of the variation of the ensemble averages of
multi-scale sign-fluctuating densities across the scales. Here, we
provide a sample computation of the ensemble averages of a
time-independent (to emphasize the spatial behavior) 1D density
$f(x)=\cos^2 (x+5) \sin \Bigl(\frac{1}{2} (x-1)^2\Bigr)$ (Figure 1),
for $R_0=10$ and $K_1=K_2=3$; its global average is approximately
$-0.003880$. The $y$-values in Figure 2 represent the ensemble
averages with respect to the $(K_1,K_2)$-covers exhibiting maximal
positive and negative bias, across the range of scales from
$10^{-2}$ to $10^1=R_0$ (the $x$-values represent the powers of
$10$); the red line slightly below the $x$-axis corresponds to the
value of the global average. The response of the ensemble averages
to sign-fluctuations at several different scales is clearly visible.

\medskip

\begin{figure}
  \centerline{\includegraphics[scale=.7]{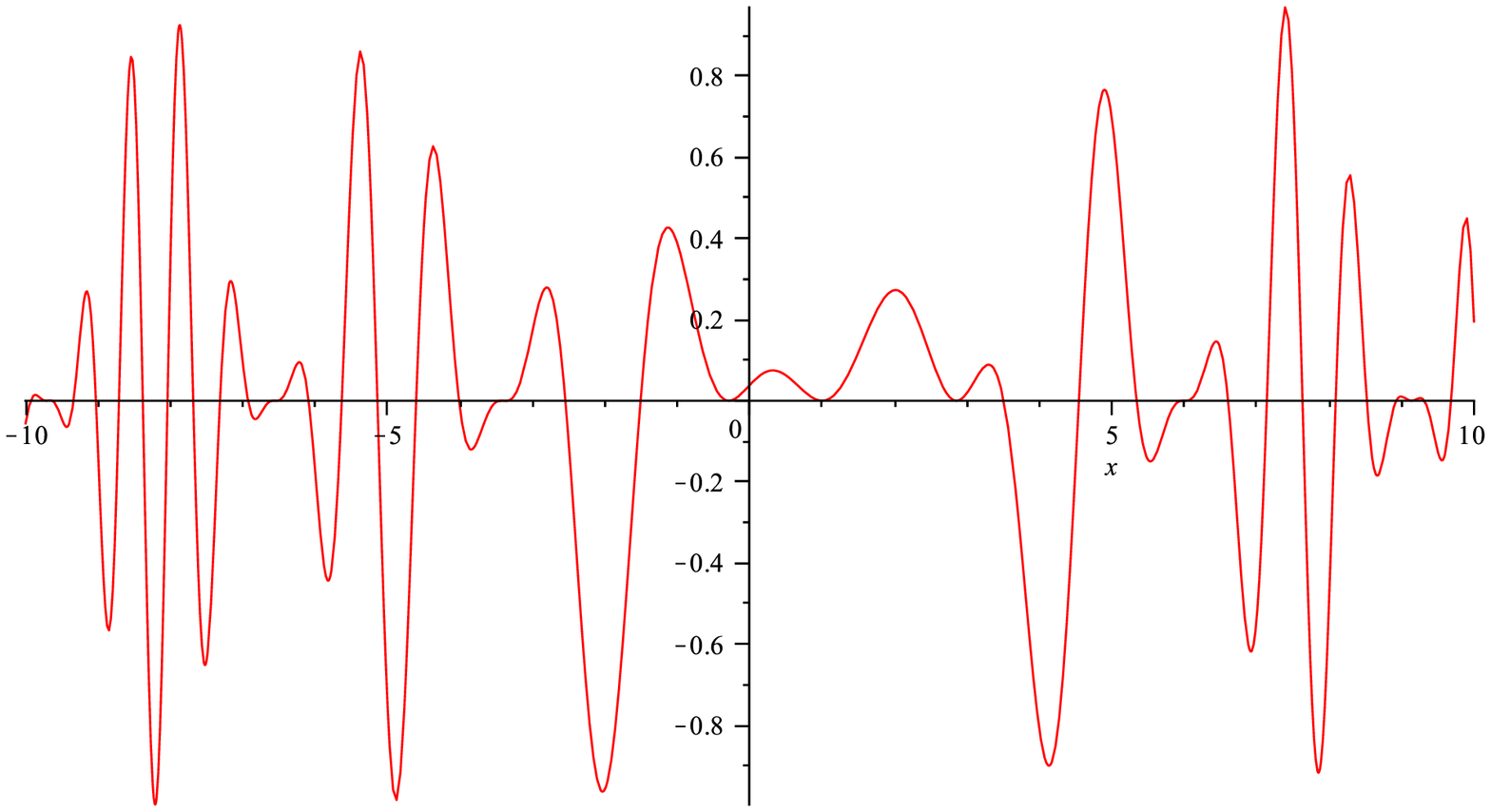}}
  \caption{$f(x)=\cos^2 (x+5) \sin \Bigl(\frac{1}{2} (x-1)^2\Bigr)$}
  \label{g_f}
\end{figure}

\medskip

\begin{figure}
  \centerline{\includegraphics[scale=.7]{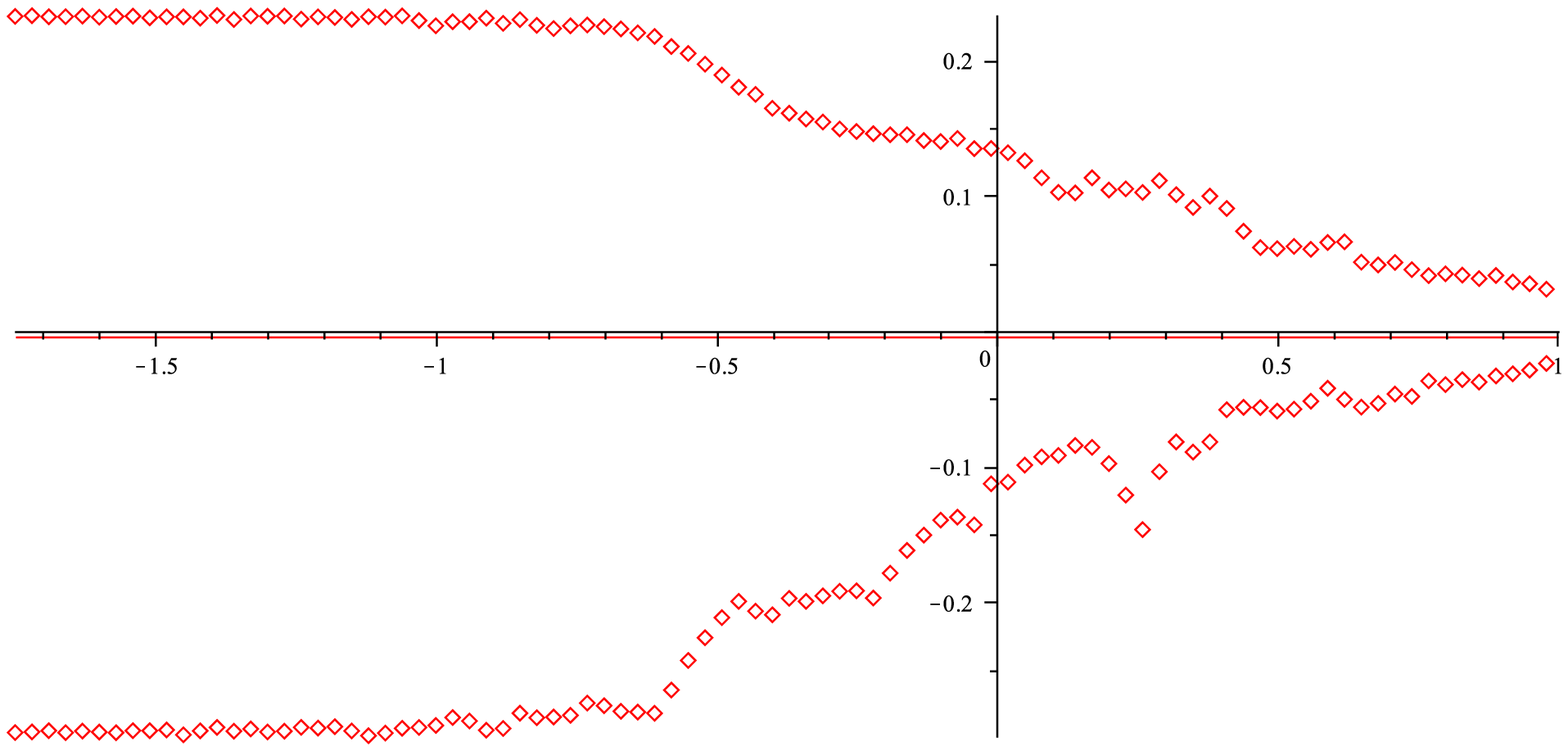}}
  \caption{Ensemble averages with positive and negative bias.}
  \label{av_pos}
\end{figure}

\section{Spatiotemporal localization of evolution of the enstrophy}

The localization of the vorticity formulation introduced in
\cite{GrZh06, Gr09} was performed on an arbitrarily small parabolic
cylinder below a point $(x_0,t_0)$ contained in the spatiotemporal
domain $\Omega \times (0,T)$, with an eye on formulating the
conditions preventing singularity formation at $(x_0,t_0)$. Here --
in order to coordinate the notation with Section 2 -- the
localization will be performed on a spatiotemporal cylinder
$B(x_0,2R) \times (0,T)$ where $x_0$ belongs to the integral domain
$B(0,R_0)$ and $0<R\le R_0$.

\medskip

Suppose that the solution is smooth in $B(x_0,2R) \times (0,T)$.
Multiplying the equations by $\phi \, \omega$ ($\phi=\psi\eta$ being
the cut-off introduced in Section 2) and integrating over $B(x_0,2R)
\times (0,t)$ for some $2T/3<t<T$ yields
\begin{align}\label{locx0}
 \int \frac{1}{2}|\omega(x,t)|^2\psi(x) \; dx &+ \int_0^t \int |\nabla\omega|^2\phi
 \;  dx \; ds\notag\\
 &= \int_0^t \int \frac{1}{2}|\omega|^2 (\phi_t+\triangle\phi) \; dx
 \; ds\notag\\
 &+ \int_0^t \int \frac{1}{2}|\omega|^2 (u \cdot \nabla\phi) \; dx
 \; ds +\int_0^t \int (\omega \cdot \nabla)u \cdot \phi \omega \; dx
 \; ds.
\end{align}

\medskip

Suppressing the time variable, the localized vortex-stretching term
can be written as (cf. \cite{Gr09})
\begin{align}\label{locvst}
  (\omega \cdot \nabla)u \cdot \phi \omega \, (x) & =
            \phi^{\frac{1}{2}}(x) \, \frac{\partial}{\partial x_i} u_j(x) \, \phi^{\frac{1}{2}}(x) \,
            \omega_i(x)\, \omega_j(x)\notag\\
  & = -c \,  P.V. \int_{B(x_0, 2r)} \epsilon_{jkl} \, \frac{\partial^2}{\partial x_i \partial y_k}
   \frac{1}{|x-y|} \, \phi^{\frac{1}{2}}  \, \omega_l \, dy \ \phi^{\frac{1}{2}}(x) \,
   \omega_i(x) \, \omega_j(x) + \ \mbox{LOT}\notag\\
  & = - c \, P.V. \int_{B_(x_0, 2r)} \bigl(\omega(x) \times
  \omega(y) \bigr) \cdot G_\omega (x,y) \, \phi^{\frac{1}{2}}(y) \, \phi^{\frac{1}{2}}(x) \, dy +
  \ \mbox{LOT}\notag\\
  & = \ \mbox{VST} \ + \ \mbox{LOT}
\end{align}
where $\epsilon_{jkl}$ is the Levi-Civita symbol,
\[
 \bigl( G_\omega (x,y) \bigr)_k = \frac{\partial^2}{\partial x_i
 \partial y_k} \frac{1}{|x-y|} \, \omega_i(x)
\]
and LOT denotes the lower order terms.

\medskip

The above representation formula for the leading order
vortex-stretching term VST features both analytic and geometric
cancelations.

\medskip

The geometric cancelations were utilized in \cite{Gr09} to obtain a
full localization of $\frac{1}{2}$-H\"older coherence of the
vorticity direction regularity criterion, and then in
\cite{GrGu10-1} to introduce a family of \emph{scaling-invariant}
regularity classes featuring a precise balance between coherence of
the vorticity direction and spatiotemporal integrability of the
vorticity magnitude. Denote by $\xi$ the vorticity direction, and
let $(x,t)$ be a spatiotemporal point, $r>0$ and $0 < \gamma < 1$. A
$\gamma$-H\"older measure of coherence of the vorticity direction at
$(x,t)$ is then given by

\[
 \rho_{\gamma, r}(x,t)=\sup_{y \in B(x, r), y \neq x}
 \frac{|\sin \varphi \bigl(\xi(x,t), \xi(y,t)\bigr)|}{|x-y|^\gamma}.
\]
The following regularity class --  a scaling-invariant improvement
of $\frac{1}{2}$-H\"older coherence -- is included,

\begin{equation}\label{hybrid0}
\int_{t_0-(2R)^2}^{t_0} \int_{B(x_0, 2R)} |\omega(x,t)|^2 \
\rho^2_{\frac{1}{2}, 2R}(x,t) dx  \, dt < \infty.
\end{equation}

\medskip

On the other hand, the analytic cancelations were utilized in
\cite{GrGu10-2} via a local non-homogeneous Div-Curl Lemma to obtain
a full spatiotemporal localization of Kozono-Taniuchi generalization
of Beale-Kato-Majda regularity criterion; namely, the
time-integrability of the $BMO$ norm of the vorticity.

\section{3D enstrophy cascade}

Let $\mathcal{R}$ be a region contained in the global spatial domain
$\Omega$. The inward enstrophy flux through the boundary of the
region is given by
\[
 - \int_{\partial\mathcal{R}} \frac{1}{2}|\omega|^2 (u\cdot n) \, d\sigma
 = - \int_\mathcal{R} (u\cdot\nabla)\omega \cdot \omega \, dx
\]
where $n$ denotes the outward normal (taking into account
incompressibility of the flow). Localization of the evolution of the
enstrophy to cylinder $B(x_0,2R) \times (0,T)$ (cf. Section 3) leads
to the following version of the enstrophy flux,
\begin{equation}\label{locflux}
\int \frac{1}{2}|\omega|^2 (u\cdot \nabla\phi) \, dx
 = - \int (u\cdot\nabla)\omega \cdot \phi \omega \, dx
\end{equation}
(again, taking into account the incompressibility; here,
$\phi=\phi_{x_0,R,T}$ defined in Section 2). Since $\nabla\phi =
(\nabla\psi) \eta$, and $\psi$ can be constructed such that
$\nabla\psi$ points inward -- toward $x_0$ -- (\ref{locflux})
represents \emph{local inward enstrophy flux, at scale $R$} (more
precisely, through the layer $S(x_0,R,2R)$) \emph{around the point
$x_0$}. In the case the point $x_0$ is close to the boundary of the
integral domain $B(0,R_0)$, $\nabla\psi$ is not exactly radial, but
still points inward.

\medskip

Consider a $(K_1,K_2)$ cover $\{B(x_i,R)\}_{i=1}^n$ at scale $R$,
for some $0<R\le R_0$. Local inward enstrophy fluxes, at scale $R$,
associated to the cover elements $B(x_i,R)$, are then given by
\begin{equation}\label{locfluxi}
\int \frac{1}{2}|\omega|^2 (u\cdot \nabla\phi_i) \, dx,
\end{equation}
for $1 \le i \le n$ ($\phi_i = \phi_{x_i,R,T}$). Assuming
smoothness, the identity (\ref{locx0}) written for $B(x_i,R)$ yields
the following expression for time-integrated local fluxes,

\begin{align}\label{loc}
 \int_0^t \int \frac{1}{2}|\omega|^2 (u \cdot \nabla\phi_i) \; dx
 \; ds
 &=
 \int \frac{1}{2}|\omega(x,t)|^2\psi_i(x) \; dx + \int_0^t \int
 |\nabla\omega|^2\phi_i
 \;  dx \; ds\notag\\
 &- \int_0^t \int \frac{1}{2}|\omega|^2 \bigl((\phi_i)_s+\triangle\phi_i\bigr) \; dx
 \; ds\notag\\
 &-\int_0^t \int (\omega \cdot \nabla)u \cdot \phi_i \, \omega \; dx
 \; ds,
\end{align}
for any $t$ in $(2T/3,T)$ and $1 \le i \le n$.

\medskip

Denoting the time-averaged local fluxes per unit mass associated to
the cover element $B(x_i,R)$ by $\hat{\Phi}_{x_i,R}$,
\begin{equation}\label{locfluxiav}
\hat{\Phi}_{x_i,R} = \frac{1}{t} \int_0^t \frac{1}{R^3} \int
\frac{1}{2}|\omega|^2 (u\cdot \nabla\phi_i) \, dx,
\end{equation}
the main quantity of interest is the ensemble average of
$\{\hat{\Phi}_{x_i,R}\}_{i=1}^n$ in the sense of Section 2; namely,
\begin{equation}\label{PhiR}
 \langle\Phi\rangle_R = \frac{1}{n}\sum_{i=1}^n \hat{\Phi}_{x_i,R}.
\end{equation}

\medskip

Since the flux density, $-(u \cdot \nabla)\omega \cdot \omega$, is
an \emph{a priori} sign-varying density, the stability, i.e., the
near-constancy of $\langle\Phi\rangle_R$ -- while the ensemble
averages are being run over all $(K_1,K_2)$ covers at scale $R$ --
will indicate there are no significant sign-fluctuations of the flux
density at the scales comparable or greater than $R$.

\medskip

The main goal of this section is to formulate a set of physically
reasonable conditions on the flow in $B(0,2R_0) \times (0,T)$
implying the positivity and near-constancy of $\langle\Phi\rangle_R$
across a suitable range of scales -- \emph{existence of the
enstrophy cascade}.

\medskip

We will consider the case $\Omega = \mathbb{R}^3$. The main reason
is that in the case of a domain with the boundary, it is necessary
to utilize the full spatial localization (joint in $x$ and $y$) of
the vortex-stretching term given by (\ref{locvst}); this introduces
a number of the lower order terms which -- in turn -- lead to terms
of the form
\[
 c_\gamma \frac{1}{R^\gamma} \iint |\omega|^2 \phi_i^{2\rho - 1} \,
 dx \, ds
\]
for some $\gamma > 2$ ($c_\gamma$ is a suitable dimensional constant
-- scaling like $R^{\gamma-2}$). Such terms introduce a correction
to the dissipation cut-off in the enstrophy cascade; namely, the
modified Kraichnan scale $\sigma_0$ (see (A2) below) would have to
be replaced by $\sigma_0^\frac{2}{\gamma}$ (times a dimensional
constant). On the other hand, in $\mathbb{R}^3$, the full (spatial)
localization of the vortex-stretching term can be replaced by the
spatial localization in $x$ only,
\begin{equation}\label{vsti}
  (\omega \cdot \nabla)u \cdot \phi_i \omega \, (x) =
   - c \, P.V. \int  \bigl(\omega(x) \times
  \omega(y) \bigr) \cdot G_\omega (x,y) \, \phi_i^{\frac{1}{2}}(x) \,
  \phi_i^{\frac{1}{2}}(x) \;
  dy ;
\end{equation}
the $y$ integral is then split in suitable small and large scales
(similarly to \cite{GrZh06}) without introducing correction terms
(cf. the proof of the main result in this section).

\bigskip

\noindent \textbf{(A1) \, Coherence Assumption}

\medskip

\noindent Denote the vorticity direction field by $\xi$, and let
$M>0$ (large). Assume that there exists a positive constant $C_1$
such that
\[
 |\sin\varphi\bigl(\xi(x,t),\xi(y,t)\bigr)| \le C_1
 |x-y|^\frac{1}{2}
\]
for any $(x,y,t)$ in $\bigl(B(0,2R_0) \times
B(0,2R_0+R_0^\frac{2}{3}) \times (0,T)\bigr) \cap \{|\nabla u| >
M\}$ ($\varphi(z_1,z_2)$ denotes the angle between the vectors $z_1$
and $z_2$). Shortly, $\frac{1}{2}$-H\"older coherence in the region
of intense fluid activity (large gradients).

\medskip

Note that the previous local regularity results \cite{GrZh06, Gr09}
imply that --  under (A1) -- the \emph{a priori} weak solution in
view is in fact smooth inside $B(0,2R_0) \times (0,T)$, and can,
moreover, be smoothly continued (locally-in-space) past $t=T$; in
particular, we can write (\ref{loc}) with $t=T$.

\medskip

Let us briefly remark that in the aforementioned works on
regularity, the region of intense fluid activity is usually defined
as $\{|\omega| > M\}$ rather that $\{|\nabla u| > M\}$. Cutting-off
at $|\omega| = M$ here would eventually lead to replacing $E_0$ in
the definition of the modified Kraichnan scale $\sigma_0$ (see the
next paragraph) by
\[
 E_0=\frac{1}{T}\int \frac{1}{R_0^3} \int \frac{1}{2}|\nabla u|^2
 \phi_0^{2\rho-1} \, dx \, dt,
\]
and we prefer keeping $\sigma_0$ solely in terms of $\omega$.

\bigskip

\noindent \textbf{(A2) \, Modified Kraichnan Scale}

\medskip

\noindent Denote by $E_0$ time-averaged enstrophy per unit mass
associated with the integral domain $B(0,2R_0) \times (0,T)$,
\[
 E_0=\frac{1}{T}\int \frac{1}{R_0^3} \int \frac{1}{2}|\omega|^2
 \phi_0^{2\rho-1} \, dx \, dt,
\]
by $P_0$ a modified time-averaged palinstrophy per unit mass,
\[
 P_0= \frac{1}{T}\int \frac{1}{R_0^3} \int |\nabla\omega|^2
 \phi_0 \, dx \, dt
 + \frac{1}{T}\frac{1}{R_0^3} \int \frac{1}{2}|\omega(x,T)|^2
 \psi_0(x) \, dx
\]
(the modification is due to the shape of the temporal cut-off
$\eta$), and by $\sigma_0$ a corresponding modified Kraichnan scale,
\[
 \sigma_0=\biggl(\frac{E_0}{P_0}\biggr)^\frac{1}{2}.
\]

\medskip

Then, the assumption (A2) is simply a requirement that the modified
Kraichnan scale associated with the integral domain $B(0,2R_0)
\times (0,T)$ be dominated by the integral scale,
\[
 \sigma_0 < \beta R_0,
\]
for a constant $\beta$, $0 < \beta < 1$, $\beta = \beta
(\rho,K_1,K_2,M,B_T)$, where $\displaystyle{B_T = \sup_{t \in (0,T)}
\|\omega(t)\|_{L^1}}$; this is finite provided the initial vorticity
is a finite Radon measure \cite{Co90}.

\bigskip

\noindent \textbf{(A3) \, Localization and Modulation}

\medskip

\noindent The general set up considered is one of the weak Leray solutions
satisfying (A1). As already mentioned, (A1) implies smoothness; however,
the control on regularity-type norms is only local. On the other hand,
the energy inequality on the global spatiotemporal domain
$\mathbb{R}^3 \times (0,T)$ implies
\[
 \int_0^T \int_{\mathbb{R}^3} |\omega|^2 \, dx \, dt < \infty;
\]
consequently, for a given constant $C_2 > 0$, there exists $R_0^* >
0$ ($R_0^* \le \min \{\sqrt{T},1\}$; this is mainly for convenience)
such that
\begin{equation}\label{a3.1}
 \int_0^T \int_{B(0,2R_0+R_0^\frac{2}{3})} |\omega|^2 \, dx \, dt
 \le \frac{1}{C_2}
\end{equation}
for any $0 < R_0 \le R_0^*$. This is the localization assumption on
$R_0$; the precise value of the constant $C_2$ is given in the proof
of the theorem -- right after the inequality (\ref{4}).

\medskip

The modulation assumption on the evolution of local enstrophy on
$(0,T)$ -- consistent with the choice of the temporal cut-off $\eta$
-- reads

\[
\int |\omega(x,T)|^2 \psi_0(x) \, dx \ge \frac{1}{2} \sup_{t \in
(0,T)} \, \int |\omega(x,t)|^2 \psi_0(x) \, dx.
\]

\medskip

\begin{obs}
\emph{Assumption (A1) is simply a quantification of the degree of
local coherence of the vorticity direction -- a manifestation of the
local `quasi 2D-like' behavior of turbulent flows -- needed to
sufficiently deplete the nonlinearity. Assumption (A2) is a slight
modification of the condition implying existence of the enstrophy
cascade in 2D (cf. \cite{DaGr11-2}); namely, a requirement that the
modified Kraichnan scale be dominated by the integral scale. Once
unraveled, it postulates that the region of interest exhibits large
vorticity gradients (relative to the vorticity magnitude, in the
spatiotemporal average) -- indicating high spatial complexity of the
flow -- and is the enstrophy analogue of the condition implying
existence of the energy cascade in 3D (\cite{DaGr11-1}), i.e., large
velocity gradients (relative to the velocity magnitude, in the
spatiotemporal average). The last assumption is somewhat technical;
however, its purpose within our theory is physical -- to prevent
uncontrolled temporal fluctuations of $R_0$-scale enstrophy. Such
fluctuations would inevitably prevent the cascade over $B(0,R_0)$
(the region of interest).}
\end{obs}

\medskip

\begin{thm}\label{cascade}
Let $u$ be a Leray solution on $\mathbb{R}^3 \times (0,T)$ with the
initial vorticity $\omega_0$ being a finite Radon measure. Suppose
that $u$ satisfies (A1)-(A3) on the spatiotemporal integral domain
$B(0, 2R_0 + R_0^\frac{2}{3}) \times (0,T)$. Then,
\[
 \frac{1}{4K_*} P_0 \le \langle\Phi\rangle_R \le 4K_* \ P_0
\]
for all $R, \, \frac{1}{\beta}\sigma_0 \le R \le R_0$ ($K_* > 1$ is
the constant in (\ref{k*})).
\end{thm}

\begin{proof}

Throughout the proof, $(K_1,K_2)$ cover parameters $\rho, K_1$ and
$K_2$, as well as the gradient cut-off $M$ will be fixed.
Henceforth, the quantities depending only on $\rho,K_1,K_2,M,B_T$
will be considered constants and denoted by a generic $K$ (that may
change from line to line).

\medskip

Let $0 < R \le R_0$. As already noted, we can write the expression
for a time-integrated local flux corresponding to the cover element
$B(x_i,R)$, (\ref{loc}), for $t=T$,

\begin{align}\label{locT}
 \int_0^T \int \frac{1}{2}|\omega|^2 (u \cdot \nabla\phi_i) \; dx
 \; ds
 &=
 \int \frac{1}{2}|\omega(x,T)|^2\psi_i(x) \; dx + \int_0^T \int
 |\nabla\omega|^2\phi_i
 \;  dx \; ds\notag\\
 &- \int_0^T \int \frac{1}{2}|\omega|^2 \bigl((\phi_i)_s+\triangle\phi_i\bigr) \; dx
 \; ds\notag\\
 &-\int_0^T \int (\omega \cdot \nabla)u \cdot \phi_i \, \omega \; dx
 \; ds.
\end{align}

\medskip

The last two terms on the right-hand side need to be estimated.

\medskip

For the first term, the properties of the cut-off $\phi_i$ together
with the condition $T \ge R_0^2 \ge R^2$ yield

\begin{equation}\label{1}
\int_0^T \int \frac{1}{2}|\omega|^2
\bigl((\phi_i)_s+\triangle\phi_i\bigr) \; dx \; ds \le K
\frac{1}{R^2} \int_0^T \int |\omega|^2 \phi_i^{2\rho-1} \; dx \; ds.
\end{equation}

\medskip

For the second term, the vortex-stretching term
\[
\int_0^T \int (\omega \cdot \nabla)u \cdot \phi_i \, \omega \; dx
 \; ds,
\]
the integration is first split into the regions in which $|\nabla u|
\le M$ and $|\nabla u| > M$. In the first region, the integral is
simply dominated by
\begin{equation}\label{2}
K \frac{1}{R^2} \int_0^T \int |\omega|^2 \phi_i^{2\rho-1} \; dx \;
ds
\end{equation}
($R \le R_0 \le 1$). In the second region, and for a fixed $x$, we
divide the domain of integration in the representation formula
(\ref{vsti}),
\[
  (\omega \cdot \nabla)u \cdot \phi_i \omega \, (x) =
   - c \, P.V. \int  \bigl(\omega(x) \times
  \omega(y) \bigr) \cdot G_\omega (x,y) \, \phi_i^{\frac{1}{2}}(x) \,
  \phi_i^{\frac{1}{2}}(x) \,
  dy,
\]
into the regions outside and inside the sphere $\{y: \,
|x-y|=R^\frac{2}{3}\}$. In the first case, the integral is bounded
by
\begin{align}\label{3}
\iint_{\{|\nabla u|>M\}} & \int_{\{y: \, |x-y| > R^\frac{2}{3}\}}
\frac{1}{|x-y|^3} |\omega| \, dy \ \phi_i |\omega|^2 \, dx \,
ds\notag\\
& \le K \frac{1}{R^2} \, \sup_t \, \|\omega(t)\|_{L^1} \int_0^T \int
|\omega|^2 \phi_i^{2\rho-1} \; dx \; ds\notag\\
& \le K \frac{1}{R^2} \int_0^T \int |\omega|^2 \phi_i^{2\rho-1} \;
dx \; ds.
\end{align}
In the second case -- utilizing (A1), the Hardy-Littlewood-Sobolev
and the Gagliardo-Nirenberg interpolation inequalities, and the
localization part of (A3) -- the following string of bounds
transpires.
\begin{align}\label{4}
\iint_{\{|\nabla u|>M\}} & \int_{\{y: \, |x-y| < R^\frac{2}{3}\}}
\frac{1}{|x-y|^\frac{5}{2}} |\omega| \, dy \ \phi_i |\omega|^2 \, dx
\,
ds\notag\\
& \le K \int_0^T \|\omega\|_{L^2\bigl(B(x_0,
2R_0+R_0^\frac{2}{3})\bigr)}
\| \, |\phi_i^\frac{1}{2} \omega|^2 \, \|_\frac{3}{2} \, ds \notag\\
& \le K \int_0^T \|\omega\|_{L^2\bigl(B(x_0,
2R_0+R_0^\frac{2}{3})\bigr)} \|\phi_i^\frac{1}{2} \omega\|_2 \,
\|\nabla(\phi_i^\frac{1}{2}
\omega)\|_2 \, ds\notag\\
& \le K \ \biggl(\int_0^T \|\omega\|_{L^2\bigl(B(x_0,
2R_0+R_0^\frac{2}{3})\bigr)}^2 \, ds\biggr)^\frac{1}{2} \,
\biggl(\frac{1}{2} \sup_{t \in (0,T)} \|\psi_i^\frac{1}{2}
\omega\|_2^2 + \int_0^T \|\nabla(\phi_i^\frac{1}{2} \omega)\|_2^2 \,
ds\biggr)\notag\\
& \le \frac{1}{8K_*^2} \ \biggl(\frac{1}{2} \sup_{t \in (0,T)}
\|\psi_i^\frac{1}{2} \omega\|_2^2 + \int_0^T
\|\nabla(\phi_i^\frac{1}{2}
\omega)\|_2^2 \, ds\biggr),\notag\\
\end{align}
where $K_*$ is the constant in (\ref{k*}) (in the last line, we used
the localization assumption (\ref{a3.1}) with $C_2 = 64 K^2 K_*^4$).
Incorporating the estimate
\begin{align}\label{5}
\int |\nabla(\phi_i^\frac{1}{2} & \omega)|^2 \, dx\notag\\
& \le 2 \int |\nabla\omega|^2 \phi_i \, dx + c \int
\biggl(\frac{|\nabla\phi_i|}{\phi_i^\frac{1}{2}}\biggr)^2 \,
|\omega|^2 \, dx\notag\\
& \le 2 \int |\nabla\omega|^2 \phi_i \, dx + K \frac{1}{R^2}
\int |\omega|^2 \phi_i^{2\rho-1} \, dx\notag\\
\end{align} in
(\ref{4}), we arrive at the final bound,
\begin{align}\label{6}
\iint_{\{|\nabla u|>M\}} & \int_{\{y: \, |x-y| < R^\frac{2}{3}\}}
\frac{1}{|x-y|^\frac{5}{2}} |\omega| \, dy \ \phi_i |\omega|^2 \, dx
\, ds\notag\\
& \le \frac{1}{4K_*^2} \ \biggl(\frac{1}{2} \sup_{t \in (0,T)}
\|\psi_i^\frac{1}{2} \omega\|_2^2 + \int_0^T \| \nabla\omega \,
\phi_i^\frac{1}{2}\|_2^2 \, ds\biggr) + K \frac{1}{R^2} \int
|\omega|^2 \phi_i^{2\rho-1} \, dx.\notag\\\notag\\
\end{align}

Collecting the estimates (\ref{1})-(\ref{6}) and using the
modulation part of (A3), the relation (\ref{locT})
yields
\begin{align}\label{locTT}
 \int_0^T \int \frac{1}{2}|\omega|^2 (u \cdot \nabla\phi_i) \; dx
 \; ds
 &=
 \int \frac{1}{2}|\omega(x,T)|^2\psi_i(x) \; dx + \int_0^T \int
 |\nabla\omega|^2\phi_i
 \;  dx \; ds \, + \, Z\notag\\
 \end{align}
where
\[
 Z \le \frac{1}{2K_*^2} \ \biggl(\frac{1}{2}
\|\psi_i^\frac{1}{2}(\cdot) \omega(\cdot,T)\|_2^2 + \int_0^T \| \nabla\omega \,
\phi_i^\frac{1}{2}\|_2^2 \, ds\biggr) + K \frac{1}{R^2} \int
|\omega|^2 \phi_i^{2\rho-1} \, dx.
\]

Taking the ensemble averages and exploiting (\ref{k*}) multiple
times, we arrive at
\[
 \frac{1}{4K_*} P_0 \le \langle \Phi \rangle_R \le 4K_* P_0
\]
for all $\frac{1}{\beta} \sigma_0 \le R \le R_0$, and a suitable
$\beta=\beta(\rho,K_1,K_2,M,B_T)$.
\end{proof}

\begin{obs}
\emph{The first mathematical result on existence of 2D enstrophy
cascade is in the paper by Foias, Jolly, Manley and Rosa
\cite{FJMR02}; the general setting is the one of infinite-time
averages in the space-periodic case, and the cascade is in the
Fourier space, i.e., in the wavenumbers. A recent work
\cite{DaGr11-3} provides existence of 2D enstrophy cascade in the
physical space utilizing the general mathematical setting for the
study of turbulent cascade in physical scales introduced in
\cite{DaGr11-1}. To the best of our knowledge, the present paper is
the first rigorous result concerning existence of the enstrophy
cascade in 3D.}
\end{obs}

\medskip

The second theorem concerns \emph{locality} of the flux. According
to turbulence phenomenology, the average flux at scale $R$ --
throughout the inertial range -- is supposed to be well-correlated
only with the average fluxes at nearby scales. In particular, the
locality along the \emph{dyadic scale} is expected to propagate
\emph{exponentially}.

\medskip

Denoting the time-averaged local fluxes associated to the cover
element $B(x_i,R)$ by $\hat{\Psi}_{x_i,R}$,
\begin{equation}\label{locfluxiavv}
\hat{\Psi}_{x_i,R} = \frac{1}{T} \int_0^T  \int
\frac{1}{2}|\omega|^2 (u\cdot \nabla\phi_i) \, dx,
\end{equation}
the (time and ensemble) averaged flux is given by
\begin{equation}\label{PsiR}
 \langle\Psi\rangle_R = \frac{1}{n}\sum_{i=1}^n \hat{\Psi}_{x_i,R} =
 R^3 \, \langle\Phi\rangle_R.
\end{equation}

\medskip

The following locality result is a simple consequence of the
universality of the cascade of the time and ensemble-averaged local
fluxes \emph{per unit mass} $\langle\Phi\rangle_R$ obtained in
Theorem 4.1.

\begin{thm}\label{locality}
Let $u$ be a Leray solution on $\mathbb{R}^3 \times (0,T)$ with the
initial vorticity $\omega_0$ being a finite Radon measure. Suppose
that $u$ satisfies (A1)-(A3) on the spatiotemporal integral domain
$B(0, 2R_0 + R_0^\frac{2}{3}) \times (0,T)$, and let $R$ and $r$ be
two scales within the inertial range delineated in Theorem 4.1. Then
\[
\frac{1}{16K_*^2} \biggl(\frac{r}{R}\biggr)^3 \le \frac{\langle \Psi
\rangle_r}{\langle \Psi \rangle_R} \le 16K_*^2
\biggl(\frac{r}{R}\biggr)^3.
\]
In particular, if $r=2^k R$ for some integer $k$, i.e., through the
dyadic scale,
\[
\frac{1}{16K_*^2} \ 2^{3k} \le \frac{\langle \Psi \rangle_{2^k
R}}{\langle \Psi \rangle_R} \le 16K_*^2 \ 2^{3k}.
\]
\end{thm}

\begin{obs}
\emph{Previous locality results include locality of the flux via a
smooth filtering approach presented in \cite{E05} (see also
\cite{EA09}), and locality of the flux in the Littlewood-Paley
setting obtained in \cite{CCFS08}. The aforementioned results are
derived independently of existence of the inertial range, and are
essentially \emph{kinematic upper bounds} on the localized
flux in terms of a
suitable physical quantity localized to the nearby scales; the
corresponding lower bounds hold assuming saturation of certain
inequalities consistent with the turbulent behavior. In
contrast, our result is derived \emph{dynamically} as a direct
consequence of existence of the turbulent cascade in view, and
features \emph{comparable upper and lower bounds} throughout the
inertial range.}
\end{obs}

\bigskip

\bigskip

\noindent ACKNOWLEDGEMENTS \ The authors thank an anonymous referee
for a number of suggestions that lead to the present version of the
paper. Z.G. acknowledges the support of the \emph{Research Council of
Norway} via the grant number 213473 - FRINATEK.

\end{document}